\documentclass[pdflatex,sn-mathphys-num]{sn-jnl}
\usepackage{amsmath,amssymb,amsthm}
\usepackage{graphicx}
\usepackage{booktabs}
\usepackage{cleveref}
\usepackage{enumitem}
\usepackage{xcolor}
\usepackage{array}
\usepackage{etoolbox}
\usepackage{placeins}
\hypersetup{hidelinks}
\makeatletter
\patchcmd{\@maketitle}
  {Corresponding author(s). E-mail(s):}
  {Corresponding author. E-mail:}
  {}{\ClassError{sn-jnl}{Could not update corresponding-author labels}{}}
\makeatother

\theoremstyle{thmstyleone}
\newtheorem{theorem}{Theorem} 
 
\newtheorem{assumption}{Assumption} 
\newtheorem{corollary}[theorem]{Corollary} 
 
\newtheorem{proposition}[theorem]{Proposition} 
\newtheorem{algorithm}{Algorithm}

\theoremstyle{thmstyletwo}

\newtheorem{example}{Example}

\theoremstyle{thmstylethree}

\crefname{assumption}{Assumption}{Assumptions}
\Crefname{assumption}{Assumption}{Assumptions}
\crefname{algorithm}{Algorithm}{Algorithms}
\Crefname{algorithm}{Algorithm}{Algorithms}

\newcommand{\R}{\mathbb{R}}
\newcommand{\norm}[1]{\left\lVert#1\right\rVert}
\newcommand{\ip}[2]{\left\langle#1,#2\right\rangle}
\newcommand{\vect}[1]{\boldsymbol{#1}}
\newcommand{\mat}[1]{\mathbf{#1}}

\begin{document}
\raggedbottom
\allowdisplaybreaks[0]

\title{Cost-Aware Iteration Allocation for Sequential Multi-Fidelity First-Order Optimization}

\author[2]{\fnm{Amir} \sur{Ardestani-Jaafari}}
\author*[1]{\fnm{Warren} \sur{Hare}}\email{warren.hare@ubc.ca}
\author[1]{\fnm{Zhongda} \sur{Huang}}

\affil[1]{\orgdiv{Department of Mathematics},
\orgname{University of British Columbia},
\orgaddress{\street{3187 University Way}, \city{Kelowna},
\state{BC}, \postcode{V1V 1V7}, \country{Canada}}}
\affil[2]{\orgdiv{Faculty of Management},
\orgname{University of British Columbia},
\orgaddress{\street{1137 Alumni Avenue}, \city{Kelowna},
\state{BC}, \postcode{V1V 1V7}, \country{Canada}}}

\abstract{
Optimization can be expensive when evaluating the objective and its gradient
requires a large dataset or a costly simulation. A cheaper approximation
(i.e., a surrogate) can provide a useful initial point, but the work spent
on that surrogate must be justified by the saving on the original problem.
We study how many iterations to perform on each surrogate before switching
to a more accurate objective. For a given sequence of smooth, strongly convex objectives,
we use convergence bounds and distances between minimizers to construct a
model of the total iteration cost. We derive closed-form two-level budgets
for Nesterov's accelerated gradient method and a convex reformulation for
multiple levels. We also develop an online switching rule that updates the
cost comparison using the gradient at the current point. We test the
techniques on a logistic-regression problem using HIGGS data.
The online rule reduces weighted iteration cost in all 27 tested
configurations, with a median cost of 52.8\% of direct optimization.
The results illustrate why updating the switching decision can be more
effective than fixing the budget from an initial convergence bound.}


\maketitle

\section{Introduction}
An optimization method may require many evaluations of an objective and
its gradient before it finds an accurate solution. When each evaluation
involves a large dataset or a costly simulation, these evaluations can
dominate the computational cost. One way to reduce this cost is to begin
with a cheaper approximation of the objective. For example, an empirical
loss can be evaluated on a subset of the data, or a simulation can use a
coarser discretization. We call the original objective the
\emph{high-fidelity objective} and its cheaper approximation a
\emph{low-fidelity objective}.

A solution of the low-fidelity problem need not solve the high-fidelity
problem, but it may provide a better initial point. Using the output of
one optimization problem as the initial point for another is called a
\emph{warm start}. We consider a sequential use of warm starts. An
optimization method first performs some iterations on the low-fidelity
objective and then switches to the high-fidelity objective, starting from
the point it has reached. The same idea can be applied to a sequence of
objectives of increasing fidelity.

The main question is when to switch. Switching too early may give little
improvement in the initial point. Switching too late may waste iterations
solving an approximation whose minimizer differs from the solution we
actually need. A successful choice must account for both the cost per
iteration and the reduction in later work. We refer to the number of
iterations assigned to an objective as its \emph{iteration budget}, and
to the choice of these budgets as \emph{iteration allocation}.

In this paper, we study this choice for a given sequence of smooth,
strongly convex objectives. Each objective has a unique minimizer, and
we assume bounds on the distances between minimizers at consecutive
levels. Convergence bounds then describe how the iterations at one
level affect the initial error at the next. Combining these bounds with
the cost per iteration gives a model for the total cost of reaching a
specified accuracy on the high-fidelity objective. The model also lets
us compare a warm start with optimizing the high-fidelity objective
directly from the original initial point.

We first derive the cost model for general polynomial convergence bounds.
For two levels, allowing real-valued budgets gives a strictly convex
problem with a unique minimizer when the initial low-fidelity error is
nonzero. We then specialize the analysis to Nesterov's accelerated
gradient method (AGD). This gives closed-form budgets for two levels and
a convex reformulation for multiple levels. These results explain how
iteration costs, convergence rates, and distances between minimizers
affect the allocation.

The minimizers are usually unknown, so the theoretical formulas cannot
be used directly in real-world applications. Strong convexity provides upper bounds on the
distance to a minimizer from the gradient at a known point. We use these
bounds in two ways. The first fixes the low-fidelity budget at the
initial point. The second updates the cost comparison after each
low-fidelity iteration and switches when further low-fidelity work is
no longer favored by the model. We test both choices on a suite of 27
logistic-regression problems built from one sample of the HIGGS data. The initial fixed budgets
continue low-fidelity optimization after it provides little further
improvement to the high-fidelity initial point. They cost more than
direct AGD in all 27 problems. The online rule switches
earlier and reduces weighted iteration cost in every configuration,
with a median cost of 52.8\% of direct AGD.

\Cref{sec:related_work} reviews related work, and \Cref{sec:setting}
states the assumptions. \Cref{sec:two_level,sec:multi_level} develop the
allocation theory. \Cref{sec:experiments} presents the experiments, and
\Cref{sec:conclusion} discusses the findings, limitations, and future work.

\section{Related Work}
\label{sec:related_work}
Multi-fidelity optimization aims to reduce the cost of solving
expensive problems by making use of cheaper
approximations~\cite{fernandez2019,peherstorfer2018}.

Allocation across fidelity levels also appears in Monte Carlo estimation,
where samples are assigned according to variances and costs, either to
meet an error tolerance or to minimize the error for a given
budget~\cite{giles2008,peherstorfer2018}. In optimization,
classical trust-region frameworks establish convergence to first-order
stationary points of the high-fidelity problem by requiring the
low-fidelity model to match the value and gradient of the high-fidelity
model at each iterate~\cite{alexandrov1998,alexandrov2001}; when
high-fidelity derivatives are unavailable, fully linear models replace
this condition~\cite{march2012}.

Leary, Bhaskar, and Keane \cite{leary2003} use cheap-model outputs to build
neural-network and kriging approximations of expensive functions. Zhan
and Xing \cite{zhan2020} review expected-improvement methods, including
extensions that select evaluations across multiple fidelities.

Wu, Mader, and Martins \cite{wu2022} develop sequential gradient-based
multi-fidelity optimization for multidisciplinary design without a global
surrogate. Werner, Overton, and Peherstorfer \cite{werner2023} use
low-fidelity stages to initialize high-fidelity gradient sampling for
nonsmooth, nonconvex robust control.
\begin{table}[!b]
 \centering
 \caption{Use of low-fidelity information and basis for allocating work
 in selected related methods.}
 \label{tab:related_allocation}
 \small\setlength{\tabcolsep}{4pt}\renewcommand{\arraystretch}{1.22}
 \begin{tabular}{@{}
  >{\raggedright\arraybackslash}p{0.23\linewidth}
  >{\raggedright\arraybackslash}p{0.33\linewidth}
  >{\raggedright\arraybackslash}p{\dimexpr0.44\linewidth-16pt\relax}@{}}
  \toprule
  Method & Use of low-fidelity information & Basis for allocating work \\
  \midrule
  Mokhtari and Ribeiro~\cite{mokhtari2017} &
  Warm starts across increasing sample sizes &
  Statistical accuracy of each sampled problem \\
  \midrule
  Wu et al.~\cite{wu2022} &
  Sequential optimization with adaptive fidelity selection &
  Model-error-based stopping and error reduction per cost \\
  \midrule
  Werner et al.~\cite{werner2023} &
  Warm starts across gradient-sampling stages &
  Stationarity and sampling-radius tolerances \\
  \midrule
  \textbf{This paper} &
  Warm starts within a prescribed objective sequence &
  Minimization of a total-cost model for a prescribed final accuracy \\
  \bottomrule
 \end{tabular}
\end{table}

Sample-size adaptation is closely related to our logistic-regression
experiment. Varying-sample methods control gradient error or sample
variance as the amount of data increases \cite{byrd2012,friedlander2012}.
Dynamic sample-size and batch-expansion methods solve increasingly large
empirical-risk problems~\cite{daneshmand2016,mokhtari2017,derezinski2018};
some solve each problem to its statistical accuracy and grow the sample
geometrically~\cite{mokhtari2017,derezinski2018}. Adaptive second-order
methods use the same principle \cite{eisen2018}. Our hierarchy need not
be defined by sample size.

Kreji\'{c} and Krklec Jerinki\'{c} \cite{krejic2019} balance approximation
and optimization errors in a stochastic spectral projected gradient
method. They combine nonmonotone line search with sample-size updates
and show, without an explicit sample-growth condition, that the sample
size grows without bound and that, almost surely, some accumulation
point is stationary. Their approximation error is measured against an
expectation objective, whereas our bounds compare minimizers of fixed
objectives.

Warm starts also reduce the cost of related subproblems in global
optimization. Kuno \cite{kuno2022} modifies the convex relaxations in a
rectangular branch-and-bound method for difference-of-convex optimization
to preserve feasibility of the preceding solution, then solves them using
Frank--Wolfe iterations.

Finally, Chen, Shi, and Yuan \cite{chen2025} analyze Nesterov acceleration
for strongly convex functions using high-resolution differential
equations and Lyapunov analysis, extending the convergence argument to
a step size equal to the inverse of the gradient's Lipschitz constant.
Their focus is the rate on a single objective.

\Cref{tab:related_allocation} compares how closely related methods
allocate work across levels.

\section{Problem Setting}
\label{sec:setting}

\subsection{Fidelity levels}
Let \(M\geq2\), and let
\[
  f^{(m)}:\R^d\rightarrow\R,\qquad m=1,\ldots,M,
\]
be a given sequence of objectives from lowest to highest fidelity, with all
objectives defined on the same space. We use bold symbols for vectors
and matrices. All norms and inner products are Euclidean.

\begin{assumption}[Smooth strong convexity]
\label{ass:smooth_sc}
For every \(m\), \(0<\mu_m\leq L_m\), and \(f^{(m)}\) is
differentiable, \(L_m\)-smooth, and \(\mu_m\)-strongly convex. This means that, for all
\(\vect{x},\vect{y}\in\R^d\),
\begin{align}
  \norm{\nabla f^{(m)}(\vect{x})-\nabla f^{(m)}(\vect{y})}
  &\leq L_m\norm{\vect{x}-\vect{y}}, \label{eq:smoothness}\\
  f^{(m)}(\vect{y})
  &\geq f^{(m)}(\vect{x})+\ip{\nabla f^{(m)}(\vect{x})}{\vect{y}-\vect{x}}
       +\frac{\mu_m}{2}\norm{\vect{y}-\vect{x}}^2. \label{eq:strong_convexity}
\end{align}
Thus \(f^{(m)}\) has a unique minimizer \(\vect{x}^{*(m)}\).
\end{assumption}

The final goal is to return \(\vect{x}\) satisfying
\begin{equation}
  f^{(M)}(\vect{x})-f^{(M)}\bigl(\vect{x}^{*(M)}\bigr)\leq\varepsilon,
  \label{eq:target}
\end{equation}
for a given \(\varepsilon>0\).
We call the difference on the left the \emph{objective gap}.

\begin{assumption}[Distances between minimizers]
\label{asm2}
For each pair of consecutive levels, a nonnegative constant \(\Delta_m\)
satisfies
\[
  \norm{\vect{x}^{*(m)}-\vect{x}^{*(m+1)}}\leq\Delta_m,
  \qquad m=1,\ldots,M-1.
\]
\end{assumption}
We require neither ordered objective values nor decreasing distances
between consecutive minimizers.
\Cref{sec:certificates} gives computable choices of \(\Delta_m\).

\subsection{Generic convergence-rate model}

Let \(\vect{x}_k^{(m)}\) denote the point obtained after \(k\) iterations of a chosen optimization method on \(f^{(m)}\), initialized at \(\vect{x}_0^{(m)}\).
One iteration is one update of the chosen
method, with cost \(c_m\) at level \(m\). Typically,
\[
  0<c_1\leq c_2\leq\cdots\leq c_M,
\]
although the analysis only requires positive costs.

\begin{assumption}[Polynomial convergence bounds]
\label{ass:rates}
For every level \(m\), there exist constants \(A_m,B_m>0\) and exponents
\(p,q>0\) such that, for every integer \(k\geq1\) and every initial
point \(\vect{x}_0^{(m)}\in\R^d\),
\begin{align}
 f^{(m)}\bigl(\vect{x}_k^{(m)}\bigr)
 -f^{(m)}\bigl(\vect{x}^{*(m)}\bigr)
 &\leq \frac{A_m}{k^{p}}\norm{\vect{x}_0^{(m)}-\vect{x}^{*(m)}}^2,
 \label{eq:function_rate}\\
 \norm{\vect{x}_k^{(m)}-\vect{x}^{*(m)}}
 &\leq \frac{B_m}{k^{q}}\norm{\vect{x}_0^{(m)}-\vect{x}^{*(m)}}.
 \label{eq:distance_rate}
\end{align}
\end{assumption}

The exponents \(p\) and \(q\) are the same at every level. The constants
\(A_m\) and \(B_m\) may differ because the objectives can have different
smoothness and strong-convexity constants.

\begin{example}[Accelerated gradient specialization]
\label{ex:agd_rates}
Consider Nesterov's accelerated gradient method with step size
\(1/L_m\) and the iteration-dependent momentum stated in
\Cref{subsec:experiment_algorithm}. The convergence result in
\cite{nesterov1983,beck2009} gives
\[
 f^{(m)}\bigl(\vect{x}_k^{(m)}\bigr)-f^{(m)}\bigl(\vect{x}^{*(m)}\bigr)
 \leq \frac{2L_m}{(k+1)^2}\norm{\vect{x}_0^{(m)}-\vect{x}^{*(m)}}^2
 \leq \frac{2L_m}{k^2}\norm{\vect{x}_0^{(m)}-\vect{x}^{*(m)}}^2.
\]
Combining this bound with strong convexity gives \Cref{ass:rates} with
\[
 A_m=2L_m,\qquad p=2,\qquad
 B_m=2\sqrt{\frac{L_m}{\mu_m}},\qquad q=1.
\]
\end{example}

\section{Two-Level Cost Analysis}
\label{sec:two_level}
Additional low-fidelity iterations incur a cost but may reduce the
high-fidelity work needed to reach the target accuracy. To determine
how much low-fidelity work to allocate, we first bound the number of
high-fidelity iterations required after switching. Combining this
bound with the cost of the low-fidelity iterations gives a total-cost
model. We then minimize this model over the low-fidelity budget and
compare its minimum with the cost model for direct high-fidelity
optimization.

\subsection{Iterations required after switching}
\begin{theorem}[Two-level warm-start convergence]
\label{thm:two_level_warmstart_poly}
Under \Cref{ass:smooth_sc,asm2,ass:rates}, run \(K_1\geq1\) integer
iterations on \(f^{(1)}\) from \(\vect{x}_0\). At a switch from the lower
level, initialize level 2 at \(\vect{x}_{K_1}^{(1)}\) and reset any
momentum or other algorithmic state. Define the sufficient
high-fidelity budget
\begin{equation}
 \overline K_2(K_1):=\left(\frac{A_2}{\varepsilon}\right)^{1/p}
 \left(\frac{B_1}{K_1^q}\norm{\vect{x}_0-\vect{x}^{*(1)}}
       +\Delta_1\right)^{2/p}.
 \label{eq:K2_requirement_poly}
\end{equation}
Any positive integer iteration budget \(K_2\geq\overline K_2(K_1)\) guarantees
\[
 f^{(2)}(\vect{x}_{K_2}^{(2)})-f^{(2)}(\vect{x}^{*(2)})\leq\varepsilon.
\]
\end{theorem}

\begin{proof}
The triangle inequality, \Cref{asm2}, and \eqref{eq:distance_rate} give
\begin{equation}
 \begin{aligned}
 \norm{\vect{x}_{K_1}^{(1)}-\vect{x}^{*(2)}}
 &\leq\norm{\vect{x}_{K_1}^{(1)}-\vect{x}^{*(1)}}
       +\norm{\vect{x}^{*(1)}-\vect{x}^{*(2)}}\\
 &\leq\frac{B_1}{K_1^q}\norm{\vect{x}_0-\vect{x}^{*(1)}}+\Delta_1.
 \end{aligned}
 \label{eq:warmstart_bound_poly}
\end{equation}
Applying \eqref{eq:function_rate} to the reinitialized level-2 method,
\[
 f^{(2)}(\vect{x}_{K_2}^{(2)})-f^{(2)}(\vect{x}^{*(2)})
 \leq\frac{A_2}{K_2^p}
       \left(\frac{B_1}{K_1^q}\norm{\vect{x}_0-\vect{x}^{*(1)}}
             +\Delta_1\right)^2
 \leq\varepsilon.
\]
\end{proof}

\subsection{Direct optimization and the break-even condition}
\begin{proposition}[Iterations for direct high-fidelity optimization]
\label{prop:direct_high_fidelity_cost}
Under \Cref{ass:rates}, define
\begin{equation}
 K_{\mathrm{dir}}:=\left(\frac{A_2}{\varepsilon}\right)^{1/p}
                  \norm{\vect{x}_0-\vect{x}^{*(2)}}^{2/p}.
 \label{eq:Kdir_general}
\end{equation}
Running the high-fidelity method directly from \(\vect{x}_0\) for any
integer \(k\geq\max\{1,\lceil K_{\mathrm{dir}}\rceil\}\) guarantees
\(f^{(2)}(\vect{x}_k^{(2)})-f^{(2)}(\vect{x}^{*(2)})\leq\varepsilon\).
For comparison before integer rounding, define the direct cost model by
\begin{equation}
 \mathrm{Cost}_{\mathrm{dir}}:=c_2K_{\mathrm{dir}}.
 \label{eq:direct_cost_general}
\end{equation}
The corresponding cost after integer rounding is
\(c_2\max\{1,\lceil K_{\mathrm{dir}}\rceil\}\).
\end{proposition}

\begin{proof}
For direct optimization from \(\vect{x}_0\), \eqref{eq:function_rate} gives
\[
 f^{(2)}(\vect{x}_k^{(2)})-f^{(2)}(\vect{x}^{*(2)})
 \leq\frac{A_2}{k^p}\norm{\vect{x}_0-\vect{x}^{*(2)}}^2.
\]
The right-hand side is at most \(\varepsilon\) for every integer
\(k\geq K_{\mathrm{dir}}\) with \(k\geq1\).
\end{proof}

If \(\vect{x}_0=\vect{x}^{*(2)}\), the high-fidelity problem is already
solved. In the comparison below we exclude this case so that the direct
cost and the denominator in \(\rho(K_1)\) are positive.

\begin{proposition}[Cost model using the distance after switching]
\label{prop:warmstart_cost}
Assume \(\vect{x}_0\neq\vect{x}^{*(2)}\), and let \(K_1\geq1\) be an integer.
Define the ratio of the new and original distances to the high-fidelity minimizer
\begin{equation}
 \rho(K_1):=\frac{\norm{\vect{x}_{K_1}^{(1)}-\vect{x}^{*(2)}}}
                 {\norm{\vect{x}_0-\vect{x}^{*(2)}}}.
 \label{eq:rho_def}
\end{equation}
Under \Cref{ass:rates}, any positive integer high-fidelity budget satisfying
\begin{equation}
 K_2\geq K_{\mathrm{dir}}\rho(K_1)^{2/p}
 \label{eq:K2_rho_general}
\end{equation}
guarantees an objective gap of at most \(\varepsilon\) at level 2.
Before integer rounding, the corresponding cost model is
\begin{equation}
 c_1K_1+c_2K_{\mathrm{dir}}\rho(K_1)^{2/p}.
 \label{eq:warmstart_cost_general}
\end{equation}
The cost after integer rounding is
\[
 c_1K_1+c_2\max\left\{1,
   \left\lceil K_{\mathrm{dir}}\rho(K_1)^{2/p}\right\rceil\right\}.
\]
\end{proposition}

\begin{proof}
The level-2 function-value bound is at most \(\varepsilon\) whenever
\[
 K_2\geq\left(\frac{A_2}{\varepsilon}\right)^{1/p}
         \norm{\vect{x}_{K_1}^{(1)}-\vect{x}^{*(2)}}^{2/p}
       =K_{\mathrm{dir}}\rho(K_1)^{2/p},\qquad K_2\geq1.
\]
Multiplying the required high-fidelity budget by \(c_2\) and adding
\(c_1K_1\) gives \eqref{eq:warmstart_cost_general}.
\end{proof}

\begin{theorem}[When the cost model favors a warm start]
\label{thm:warmstart_more_efficient}
Under \Cref{ass:smooth_sc,asm2,ass:rates}, assume
\(\vect{x}_0\neq\vect{x}^{*(2)}\). For a positive integer \(K_1\),
the cost model in \eqref{eq:warmstart_cost_general}
is smaller than \(\mathrm{Cost}_{\mathrm{dir}}\) if and only if
\begin{equation}
 \rho(K_1)^{2/p}+\frac{c_1}{c_2}\frac{K_1}{K_{\mathrm{dir}}}<1.
 \label{eq:WS_condition_general}
\end{equation}
A sufficient condition is
\begin{equation}
 \left(\frac{\dfrac{B_1}{K_1^q}\norm{\vect{x}_0-\vect{x}^{*(1)}}+\Delta_1}
             {\norm{\vect{x}_0-\vect{x}^{*(2)}}}\right)^{2/p}
 +\frac{c_1}{c_2}\frac{K_1}{K_{\mathrm{dir}}}<1.
 \label{eq:WS_explicit_general}
\end{equation}
\end{theorem}

\begin{proof}
Divide the inequality
\(c_1K_1+c_2K_{\mathrm{dir}}\rho(K_1)^{2/p}<c_2K_{\mathrm{dir}}\)
by \(c_2K_{\mathrm{dir}}>0\) to obtain
\eqref{eq:WS_condition_general}. By \eqref{eq:warmstart_bound_poly},
\[
 \rho(K_1)\leq
 \frac{\dfrac{B_1}{K_1^q}\norm{\vect{x}_0-\vect{x}^{*(1)}}+\Delta_1}
      {\norm{\vect{x}_0-\vect{x}^{*(2)}}},
\]
which gives \eqref{eq:WS_explicit_general}.
\end{proof}

In \eqref{eq:WS_condition_general}, the first term is the predicted fraction of
high-fidelity cost after switching, relative to direct optimization. The second is the
low-fidelity cost relative to the direct cost. Their sum must be less
than one. Equation \eqref{eq:WS_explicit_general} replaces the post-switch
distance in the numerator by the upper bound in
\eqref{eq:warmstart_bound_poly}. The initial distance in the
denominator is unchanged.
Both comparisons are between cost models before integer rounding.

\subsection{Minimizing the cost model}
Equation \eqref{eq:warmstart_bound_poly} gives an explicit upper bound
on the cost model in \eqref{eq:warmstart_cost_general}. We minimize
this upper bound to choose the low-fidelity budget. For this
continuous model, we allow \(K_1>0\) and omit integer rounding at
the high-fidelity level. The subscript \(\mathrm{WS}\) in the following
theorem denotes the warm-start strategy.

\begin{theorem}[Unique minimizer of the two-level cost model]
\label{thm:optimal_allocation_general}
Under \Cref{ass:smooth_sc,asm2,ass:rates}, assume
\(\vect{x}_0\neq\vect{x}^{*(1)}\),
\(c_1,c_2>0\), and \(\varepsilon>0\). Define
\begin{equation}
 \mathrm{Cost}_{\mathrm{WS}}(K_1)
 :=c_1K_1+c_2\left(\frac{A_2}{\varepsilon}\right)^{1/p}
       \left(\frac{B_1}{K_1^q}\norm{\vect{x}_0-\vect{x}^{*(1)}}
             +\Delta_1\right)^{2/p},
 \label{eqn:warmstartcost}
\end{equation}
for \(K_1>0\). This cost model is strictly convex, and its minimizer
\(K_1^*\) is the unique positive solution of
\begin{equation}
 c_1=c_2\left(\frac{A_2}{\varepsilon}\right)^{1/p}
 \frac{2qB_1\norm{\vect{x}_0-\vect{x}^{*(1)}}}{p(K_1^*)^{q+1}}
 \left(\frac{B_1}{(K_1^*)^q}\norm{\vect{x}_0-\vect{x}^{*(1)}}
       +\Delta_1\right)^{2/p-1}.
 \label{eq:allocation_balance}
\end{equation}
The corresponding high-fidelity budget is \(K_2^*:=\overline K_2(K_1^*)\),
evaluated from \eqref{eq:K2_requirement_poly} at the real value \(K_1^*\).
\end{theorem}

\begin{proof}
\Cref{thm:two_level_warmstart_poly} gives the required high-fidelity budget in
\eqref{eq:K2_requirement_poly}, yielding the cost model
\eqref{eqn:warmstartcost}. If the initial low-fidelity error were zero,
this model would be affine and strictly increasing in \(K_1>0\), with
no minimizer. This is why the theorem requires
\(\vect{x}_0\neq\vect{x}^{*(1)}\).
To analyze the curvature under this condition, set
\[
 a=B_1\norm{\vect{x}_0-\vect{x}^{*(1)}}>0,\qquad
 s=\frac{2}{p}>0,\qquad
 g(K)=(aK^{-q}+\Delta_1)^s.
\]
Differentiation gives
\begin{align*}
 g'(K)&=-sq aK^{-q-1}(aK^{-q}+\Delta_1)^{s-1},\\
 g''(K)&=sq aK^{-q-2}(aK^{-q}+\Delta_1)^{s-2}
          \bigl[(q+1)\Delta_1+(sq+1)aK^{-q}\bigr]>0.
\end{align*}
Every factor outside the brackets is positive because
\(s,q,a,K>0\) and \(aK^{-q}+\Delta_1>0\). Inside the brackets,
\((q+1)\Delta_1\geq0\) and \((sq+1)aK^{-q}>0\), so the bracket is
also positive. Thus \(\mathrm{Cost}_{\mathrm{WS}}''(K)
=c_2(A_2/\varepsilon)^{1/p}g''(K)>0\).
Since \(a>0\), the cost tends to infinity as \(K\downarrow0\).
The term \(c_1K\) makes it tend to infinity as \(K\to\infty\).
A minimizer therefore exists and is unique. Setting the derivative to
zero gives \eqref{eq:allocation_balance}.
\end{proof}

For integer \(K_1\geq1\), the distance bound after switching gives
\[
 c_1K_1+c_2\left(\frac{A_2}{\varepsilon}\right)^{1/p}
 \norm{\vect{x}_{K_1}^{(1)}-\vect{x}^{*(2)}}^{2/p}
 \leq\mathrm{Cost}_{\mathrm{WS}}(K_1).
\]
Thus \(\mathrm{Cost}_{\mathrm{WS}}\) bounds the cost model in
\eqref{eq:warmstart_cost_general} using only the convergence rate and
the initial distance. At its optimum, \eqref{eq:allocation_balance} balances the
cost of one more low-fidelity iteration against the predicted saving
in high-fidelity work. This minimizer is optimal for the cost model,
not necessarily for the iterations observed in an experiment.
Moreover, several quantities in the model are usually unknown in
practice. They are the initial distance
\(\norm{\vect{x}_0-\vect{x}^{*(1)}}\), the bound \(\Delta_1\) on the
distance between minimizers, and the constants \(A_2\) and \(B_1\),
which for AGD depend on the smoothness and strong-convexity constants.
Before the model can be used, the distances and smoothness constants
must be replaced by computable upper bounds, and the strong-convexity
constants by positive lower bounds. For AGD,
\eqref{eq:agd_K1_expanded} shows that \(K_1^*\) increases with the
bound on \(\norm{\vect{x}_0-\vect{x}^{*(1)}}\) and with
\(L_1L_2/\mu_1\), so loose bounds produce a large low-fidelity budget
even when further low-fidelity iterations bring the iterates no closer
to the high-fidelity minimizer. \Cref{sec:certificates} bounds the
distances using gradients, \Cref{subsec:higgs_model_properties} gives
the curvature constants used in the experiments, and
\Cref{subsec:distance_diagnostics} shows this effect.

\subsection{Closed-form allocation for AGD}
\begin{corollary}[Two-level AGD budgets minimizing the cost model]
\label{thm:optimal_allocation_no_shorthand}
Under the assumptions of \Cref{thm:optimal_allocation_general}, the AGD
rates in \Cref{ex:agd_rates} give the unique continuous minimizer
\begin{equation}
 K_1^*=
 \sqrt{\frac{c_2}{c_1}
       \sqrt{\frac{8L_1L_2}{\mu_1\varepsilon}}\,
       \norm{\vect{x}_0-\vect{x}^{*(1)}}}.
 \label{eq:agd_K1_expanded}
\end{equation}
Its high-fidelity budget and minimum model cost are
\begin{align}
 K_2^*&=
 \sqrt{\frac{2L_2}{\varepsilon}}\,\Delta_1
 +\sqrt{\frac{c_1}{c_2}
        \sqrt{\frac{8L_1L_2}{\mu_1\varepsilon}}\,
        \norm{\vect{x}_0-\vect{x}^{*(1)}}},
 \label{eq:K2_star_no_shorthand}\\
 \mathrm{Cost}_{\mathrm{WS}}^*&=
 2\sqrt{c_1c_2\sqrt{\frac{8L_1L_2}{\mu_1\varepsilon}}\,
        \norm{\vect{x}_0-\vect{x}^{*(1)}}}
 +c_2\sqrt{\frac{2L_2}{\varepsilon}}\,\Delta_1.
 \label{eq:Cost_star_no_shorthand}
\end{align}
This minimum is smaller than the direct cost model
\(c_2\sqrt{2L_2/\varepsilon}\norm{\vect{x}_0-\vect{x}^{*(2)}}\)
if and only if
\begin{equation}
 \Delta_1+2\sqrt{\frac{c_1}{c_2}
                 \sqrt{\frac{2L_1\varepsilon}{\mu_1L_2}}\,
                 \norm{\vect{x}_0-\vect{x}^{*(1)}}}
 <\norm{\vect{x}_0-\vect{x}^{*(2)}}.
 \label{eq:agd_break_even}
\end{equation}
\end{corollary}

\begin{proof}
Substituting \(p=2\), \(q=1\), and the constants from
\Cref{ex:agd_rates} into \eqref{eqn:warmstartcost} gives
\[
 \mathrm{Cost}_{\mathrm{WS}}(K_1)
 =c_1K_1+\frac{c_2}{K_1}
          \sqrt{\frac{8L_1L_2}{\mu_1\varepsilon}}\,
          \norm{\vect{x}_0-\vect{x}^{*(1)}}
       +c_2\sqrt{\frac{2L_2}{\varepsilon}}\,\Delta_1.
\]
By \Cref{thm:optimal_allocation_general}, this function is strictly
convex, so its minimizer is the unique point where its derivative
vanishes. This condition is
\[
 0=c_1-\frac{c_2}{(K_1^*)^2}
      \sqrt{\frac{8L_1L_2}{\mu_1\varepsilon}}\,
      \norm{\vect{x}_0-\vect{x}^{*(1)}},
\]
which gives \eqref{eq:agd_K1_expanded}. At this point, the first two cost terms
are equal. Substitution gives \eqref{eq:K2_star_no_shorthand} and
\eqref{eq:Cost_star_no_shorthand}. Comparing this minimum with the
direct cost and dividing both sides by \(c_2\sqrt{2L_2/\varepsilon}\)
gives \eqref{eq:agd_break_even}.
\end{proof}

For AGD, the bound \(\Delta_1\) on the distance between minimizers
adds a constant to the high-fidelity cost. It affects whether the
model favors a warm start and how many high-fidelity iterations are
required, but leaves \(K_1^*\) unchanged. The break-even criterion
\eqref{eq:agd_break_even} requires
\(\Delta_1<\norm{\vect{x}_0-\vect{x}^{*(2)}}\).
A failed model comparison does not rule out a useful warm
start. With all other quantities fixed, \(K_1^*\) is proportional to
\(\varepsilon^{-1/4}\). Thus, asking for a smaller objective gap makes
the model assign more low-fidelity iterations. \Cref{sec:experiments} examines
how these model predictions compare with observed costs.

A simple rounding rule is to evaluate
\[
 c_1K_1+c_2\max\{1,\lceil\overline K_2(K_1)\rceil\}
\]
at the two budgets
\[
 K_1=\max\{1,\lfloor K_1^*\rfloor\}
 \quad\text{and}\quad
 K_1=\max\{1,\lceil K_1^*\rceil\}.
\]
Compare the smaller of these two costs with the direct cost
\[
 c_2\max\{1,\lceil K_{\mathrm{dir}}\rceil\}.
\] This rule does not necessarily minimize the cost model over all integer budgets.

The next result quantifies the effect of distance-estimation error on
the continuous AGD cost model. We write
\(\mathrm{Cost}_{\mathrm{WS}}(K_1;D_1)\) to make its dependence on the
initial low-fidelity distance explicit.

\begin{proposition}[Robustness to distance-estimation error]
\label{prop:distance_robustness}
Suppose the assumptions of Corollary~\ref{thm:optimal_allocation_no_shorthand} hold. Let
\[
 D_1:=\norm{\vect{x}_0-\vect{x}^{*(1)}},
\]
and suppose that the allocation is computed using an estimate
\(\widehat D_1>0\), while the true distance is
\[
 D_1=\alpha\widehat D_1,\qquad \alpha>0.
\]
Define
\[
 \beta:=c_2\sqrt{\frac{8L_1L_2}{\mu_1\varepsilon}},
 \qquad
 \gamma:=c_2\sqrt{\frac{2L_2}{\varepsilon}}\Delta_1.
\]
The allocation obtained from the estimated distance is
\[
 \widehat K_1=\sqrt{\frac{\beta\widehat D_1}{c_1}},
\]
whereas the allocation minimizing the cost model under the true distance is
\[
 K_1^*(\alpha)
 =\sqrt{\frac{\beta\alpha\widehat D_1}{c_1}}
 =\sqrt{\alpha}\,\widehat K_1.
\]
Moreover, letting
\[
 A:=\sqrt{c_1\beta\widehat D_1},
\]
the relative model cost increase caused by using \(\widehat K_1\)
instead of \(K_1^*(\alpha)\) is
\[
 \frac{\mathrm{Cost}_{\mathrm{WS}}(\widehat K_1;\alpha\widehat D_1)}
 {\mathrm{Cost}_{\mathrm{WS}}(K_1^*(\alpha);\alpha\widehat D_1)}-1
 =\frac{(\sqrt{\alpha}-1)^2A}{2\sqrt{\alpha}A+\gamma}.
\]
Consequently,
\[
 \begin{aligned}
 \frac{\mathrm{Cost}_{\mathrm{WS}}(\widehat K_1;\alpha\widehat D_1)}
 {\mathrm{Cost}_{\mathrm{WS}}(K_1^*(\alpha);\alpha\widehat D_1)}-1
 &\leq\frac{(\sqrt{\alpha}-1)^2}{2\sqrt{\alpha}}\\
 &=\frac12\left(\sqrt{\alpha}+\frac{1}{\sqrt{\alpha}}\right)-1.
 \end{aligned}
\]
Equality in the upper bound holds when \(\Delta_1=0\).
\end{proposition}

\begin{proof}
From Corollary~\ref{thm:optimal_allocation_no_shorthand}, the continuous cost model
can be written as
\[
 \mathrm{Cost}_{\mathrm{WS}}(K_1;D_1)
 =c_1K_1+\frac{\beta D_1}{K_1}+\gamma.
\]
Using \(\widehat D_1\) gives
\[
 \widehat K_1=\sqrt{\frac{\beta\widehat D_1}{c_1}}.
\]
If the true distance is \(D_1=\alpha\widehat D_1\), minimizing the model
cost under this distance instead gives
\[
 K_1^*(\alpha)
 =\sqrt{\frac{\beta\alpha\widehat D_1}{c_1}}
 =\sqrt{\alpha}\,\widehat K_1.
\]
Since
\[
 c_1\widehat K_1
 =\frac{\beta\widehat D_1}{\widehat K_1}=A,
\]
the model cost of the estimated allocation under the true distance is
\[
 \mathrm{Cost}_{\mathrm{WS}}(\widehat K_1;\alpha\widehat D_1)
 =(1+\alpha)A+\gamma.
\]
At the model optimum under the true distance,
\[
 \mathrm{Cost}_{\mathrm{WS}}(K_1^*(\alpha);\alpha\widehat D_1)
 =2\sqrt{\alpha}A+\gamma.
\]
Therefore,
\[
 \begin{aligned}
 \frac{\mathrm{Cost}_{\mathrm{WS}}(\widehat K_1;\alpha\widehat D_1)}
 {\mathrm{Cost}_{\mathrm{WS}}(K_1^*(\alpha);\alpha\widehat D_1)}-1
 &=\frac{(1+\alpha)A+\gamma}{2\sqrt{\alpha}A+\gamma}-1\\
 &=\frac{(1+\alpha-2\sqrt{\alpha})A}{2\sqrt{\alpha}A+\gamma}\\
 &=\frac{(\sqrt{\alpha}-1)^2A}{2\sqrt{\alpha}A+\gamma}.
 \end{aligned}
\]
Since \(\gamma\geq0\),
\[
 \frac{(\sqrt{\alpha}-1)^2A}{2\sqrt{\alpha}A+\gamma}
 \leq\frac{(\sqrt{\alpha}-1)^2}{2\sqrt{\alpha}},
\]
which proves the result.
\end{proof}

The proposition shows that the continuous AGD allocation is robust to
small estimation error in the initial low-fidelity distance.
In particular, for \(\alpha=1+\delta\) with \(\delta\) close to zero,
\[
 \frac12\left(\sqrt{\alpha}+\frac{1}{\sqrt{\alpha}}\right)-1
 =\frac{\delta^2}{8}+O(\delta^3).
\]
Thus, although the allocation itself changes at first order with the
distance estimate, the resulting relative model cost loss is second
order in the estimation error. For example, if the true distance is
50\% larger than its estimate, so that \(\alpha=1.5\), the relative
model cost increase is at most approximately 2.06\%. Even if the true
distance is twice the estimated value, corresponding to \(\alpha=2\),
the model cost increase is at most approximately 6.07\%. The distance
bound \(\Delta_1\) further reduces these relative losses whenever
\(\Delta_1>0\) and \(\alpha\neq1\).

\subsection{Gradient bounds for unknown distances and accuracy}
\label{sec:certificates}

The allocation formulas involve unknown minimizer distances. Strong
convexity lets us bound these distances using gradients. We call such
computable upper bounds \emph{gradient bounds}. The same argument
also bounds the objective gap, giving a stopping test that does not
require the minimum objective value.

\begin{proposition}[Gradient bounds on distances and objective gaps]
\label{prop:certificates}
Under \Cref{ass:smooth_sc}, for every level \(m\) and every \(\vect{z}\in\R^d\),
\begin{align}
 \norm{\vect{z}-\vect{x}^{*(m)}}
 &\leq\frac{\norm{\nabla f^{(m)}(\vect{z})}}{\mu_m},
 \label{eq:distance_certificate}\\
 f^{(m)}(\vect{z})-f^{(m)}\bigl(\vect{x}^{*(m)}\bigr)
 &\leq\frac{\norm{\nabla f^{(m)}(\vect{z})}^2}{2\mu_m}.
 \label{eq:gap_certificate}
\end{align}
For adjacent levels, any common point \(\vect{z}\) also gives
\begin{equation}
 \norm{\vect{x}^{*(m)}-\vect{x}^{*(m+1)}}
 \leq
 \frac{\norm{\nabla f^{(m)}(\vect{z})}}{\mu_m}
 +\frac{\norm{\nabla f^{(m+1)}(\vect{z})}}{\mu_{m+1}}.
 \label{eq:mismatch_certificate}
\end{equation}
\end{proposition}

\begin{proof}
By strong convexity, the gradient satisfies
\[
 \ip{\nabla f^{(m)}(\vect{z})-\nabla f^{(m)}(\vect{x}^{*(m)})}
    {\vect{z}-\vect{x}^{*(m)}}
 \geq\mu_m\norm{\vect{z}-\vect{x}^{*(m)}}^2.
\]
Since \(\nabla f^{(m)}(\vect{x}^{*(m)})=0\), Cauchy--Schwarz gives
\eqref{eq:distance_certificate}. Minimizing the right-hand side of
\eqref{eq:strong_convexity} over \(\vect{y}\) gives
\eqref{eq:gap_certificate}. Finally,
\[
 \norm{\vect{x}^{*(m)}-\vect{x}^{*(m+1)}}
 \leq\norm{\vect{x}^{*(m)}-\vect{z}}+\norm{\vect{z}-\vect{x}^{*(m+1)}}.
\]
Applying \eqref{eq:distance_certificate} at level \(m\) bounds the
first term by \(\|\nabla f^{(m)}(\vect{z})\|/\mu_m\). Applying it at
level \(m+1\) bounds the second term by
\(\|\nabla f^{(m+1)}(\vect{z})\|/\mu_{m+1}\).
Adding these bounds gives \eqref{eq:mismatch_certificate}.
\end{proof}

Equation \eqref{eq:distance_certificate} permits the replacement
\[
 \norm{\vect{x}_0-\vect{x}^{*(1)}}
 \leq\frac{\norm{\nabla f^{(1)}(\vect{x}_0)}}{\mu_1}
\]
in \eqref{eq:agd_K1_expanded}. Equation
\eqref{eq:mismatch_certificate} provides a valid choice of
\(\Delta_m\) by summing the two gradient bounds at a common point.
Equation \eqref{eq:gap_certificate} gives the high-fidelity stopping rule
\[
 \frac{\norm{\nabla f^{(M)}(\vect{x})}^2}{2\mu_M}\leq\varepsilon
 \quad\Longrightarrow\quad
 f^{(M)}(\vect{x})-f^{(M)}\bigl(\vect{x}^{*(M)}\bigr)\leq\varepsilon.
\]

This gradient test uses an upper bound on the objective gap. The
actual gap may already be at most \(\varepsilon\) while the upper bound
is still larger. Consequently, a budget that guarantees the required
objective gap need not make this stopping test pass.

Replacing distances by upper bounds preserves the accuracy guarantee
when both the allocation and the required high-fidelity budget are recomputed.
As noted in \Cref{sec:two_level}, loose bounds can then lead to
excessive work, and the same holds for a direct-cost model built from
an upper bound on
\(\norm{\vect{x}_0-\vect{x}^{*(2)}}\).

\section{Multi-Level Extension}
\label{sec:multi_level}

The two-level distance bound extends to a sequence of
\(M\) objectives. The first stage starts at \(\vect{x}_0^{(1)}=\vect{x}_0\).
After level \(m<M\), we initialize the next level using
\[
 \vect{x}_0^{(m+1)}=\vect{x}_{K_m}^{(m)}.
\]
At each switch, momentum and other algorithmic state are reset.

\begin{theorem}[Distance bounds across multiple levels]
\label{thm:multilevel_transfer}
Under \Cref{ass:smooth_sc,asm2,ass:rates}, let the budgets at the first \(M-1\) levels
\(K_1,\ldots,K_{M-1}\) be positive integers and define
\begin{align}
 R_1&:=\norm{\vect{x}_0-\vect{x}^{*(1)}},\nonumber\\
 R_{m+1}
 &:=\frac{B_m}{K_m^q}R_m+\Delta_m,
 \label{eq:generic_transfer_recursion}
\end{align}
for \(m=1,\ldots,M-1\). Thus, each bound depends on the budgets
at the preceding levels.
Then \(\norm{\vect{x}_0^{(m)}-\vect{x}^{*(m)}}\leq R_m\) for every level. Define
the required budget at the final level
\begin{equation}
 \overline K_M(K_1,\ldots,K_{M-1})
 :=\left(\frac{A_M}{\varepsilon}\right)^{1/p}R_M^{2/p}.
 \label{eq:multilevel_terminal_budget}
\end{equation}
Any positive integer budget at the final level of at least
\(\overline K_M(K_1,\ldots,K_{M-1})\) satisfies
\eqref{eq:target}.
\end{theorem}

\begin{proof}
The distance bound holds with equality at level 1. Suppose it holds at
level \(m<M\). The initialization rule, triangle inequality, and
\eqref{eq:distance_rate} imply
\[
 \begin{aligned}
 \norm{\vect{x}_0^{(m+1)}-\vect{x}^{*(m+1)}}
 &\leq\norm{\vect{x}_{K_m}^{(m)}-\vect{x}^{*(m)}}
       +\norm{\vect{x}^{*(m)}-\vect{x}^{*(m+1)}}\\
 &\leq\frac{B_m}{K_m^q}R_m+\Delta_m=R_{m+1}.
 \end{aligned}
\]
Induction establishes the distance bound at every level. At the final
level, \eqref{eq:function_rate} gives
\[
 f^{(M)}(\vect{x}_{K_M}^{(M)})-f^{(M)}(\vect{x}^{*(M)})
 \leq\frac{A_M}{K_M^p}R_M^2\leq\varepsilon
\]
for every positive integer final-level budget at least as large as the stated value.
\end{proof}

The cost model with integer rounding at the final level is
\begin{equation}
 \sum_{m=1}^{M-1}c_mK_m
 +c_M\max\left\{1,
   \left\lceil\overline K_M(K_1,\ldots,K_{M-1})\right\rceil\right\}.
 \label{eq:multilevel_integer_cost}
\end{equation}
This expression can be minimized over positive integer budgets
\(K_1,\ldots,K_{M-1}\). The recursion shows how the error at one level
affects the initial points at later levels. For AGD, allowing positive
real budgets and omitting integer rounding gives the following convex
reformulation.

\subsection{AGD allocation}

Suppose every level uses the AGD rates in \Cref{ex:agd_rates}. The recursion
then becomes
\begin{equation}
R_{m+1}
=
\frac{\sqrt{4L_m/\mu_m}}{K_m}R_m+\Delta_m,
\qquad m=1,\ldots,M-1.
\label{eq:R_recursion_AGD}
\end{equation}

\begin{theorem}[Convex reformulation of the multi-level AGD allocation]
\label{thm:multilevel_gp}
Under the assumptions of \Cref{thm:multilevel_transfer}, use AGD at
each level and assume \(R_1>0\), so the initial point is not already
the level-1 minimizer. Allow positive real budgets at the first
\(M-1\) levels and omit integer rounding at the final level. The cost model is
\begin{equation}
\mathrm{Cost}_{\mathrm{WS}}(K_1,\ldots,K_{M-1})
=
\sum_{m=1}^{M-1}c_mK_m
+
c_M\sqrt{\frac{2L_M}{\varepsilon}}R_M,
\qquad K_m>0,
\label{eq:multi_agd_cost}
\end{equation}
where \(R_M\) is defined recursively by \eqref{eq:R_recursion_AGD}. This
model becomes a strictly convex function of
\(\vect{u}=(u_1,\ldots,u_{M-1})\in\R^{M-1}\) under the substitution
\(K_m=e^{u_m}\). It also tends to infinity as \(\|\vect{u}\|\to\infty\).
Therefore, \eqref{eq:multi_agd_cost} has a unique minimizer over positive
real budgets.
\end{theorem}

\begin{proof}
Repeated substitution in the AGD recursion gives
\[
R_M
=
R_1\prod_{j=1}^{M-1}
\frac{\sqrt{4L_j/\mu_j}}{K_j}
+
\sum_{i=1}^{M-1}\Delta_i
\prod_{j=i+1}^{M-1}
\frac{\sqrt{4L_j/\mu_j}}{K_j},
\]
where an empty product equals one and \(B_j=\sqrt{4L_j/\mu_j}\) as in
\Cref{ex:agd_rates}. Now set \(K_m=e^{u_m}\), so
\(1/K_m=e^{-u_m}\). The variables \(u_m\) can be positive, zero, or
negative. Substitution into \eqref{eq:multi_agd_cost} gives
\[
\begin{aligned}
 \mathrm{Cost}_{\mathrm{WS}}(e^{u_1},\ldots,e^{u_{M-1}})
 ={}&\sum_{m=1}^{M-1}c_me^{u_m}\\
 &+c_M\sqrt{\frac{2L_M}{\varepsilon}}
 R_1\left(\prod_{j=1}^{M-1}B_j\right)e^{-\sum_{j=1}^{M-1}u_j}\\
 &+c_M\sqrt{\frac{2L_M}{\varepsilon}}
 \sum_{i=1}^{M-1}\Delta_i
 \left(\prod_{j=i+1}^{M-1}B_j\right)e^{-\sum_{j=i+1}^{M-1}u_j}
 .
\end{aligned}
\]
Every term is a nonnegative multiple of the
exponential of a linear function, and is convex for all real
\(\vect{u}\). Moreover, the terms \(c_me^{u_m}\) contribute the
positive-definite Hessian
\[
\operatorname{diag}
\bigl(c_1e^{u_1},\ldots,c_{M-1}e^{u_{M-1}}\bigr).
\]
Hence, the transformed objective is strictly convex.

To establish coercivity, consider a nonempty sublevel set with objective
at most \(C>0\). The terms \(c_me^{u_m}\) give the upper bounds
\(u_m\leq\log(C/c_m)\). The term
\[
 a_0\exp\left(-\sum_{j=1}^{M-1}u_j\right),
 \qquad
 a_0=c_M\sqrt{2L_M/\varepsilon}\,R_1
       \prod_{j=1}^{M-1}B_j>0,
\]
also gives \(\sum_j u_j\geq\log(a_0/C)\). For each component,
\[
 u_j\geq\log(a_0/C)-\sum_{\ell\ne j}\log(C/c_\ell).
\]
Every component therefore has both a lower and an upper bound on the
sublevel set. Continuity makes these sets compact, proving that the
objective tends to infinity as \(\|\vect{u}\|\to\infty\). Together with
strict convexity, this proves existence and uniqueness of the minimizer.
\end{proof}

The problem in logarithmic variables can be solved using standard
convex-optimization methods \cite{boyd2004}. For
\(j=1,\ldots,M-1\), the bound \(R_j\) depends only on the budgets
at preceding levels. Holding the other budgets fixed,
differentiating \eqref{eq:R_recursion_AGD} gives
\[
 \frac{\partial R_M}{\partial K_j}
 =-\frac{R_j}{K_j}
   \prod_{\ell=j}^{M-1}
   \frac{\sqrt{4L_\ell/\mu_\ell}}{K_\ell}.
\]
Setting the derivative of \eqref{eq:multi_agd_cost} to zero and
multiplying by \(K_j\) gives the following equality at the
minimizing budgets,
\begin{equation}
 c_jK_j^*
 =c_M\sqrt{\frac{2L_M}{\varepsilon}}\,R_j^*
   \prod_{\ell=j}^{M-1}\frac{\sqrt{4L_\ell/\mu_\ell}}{K_\ell^*},
 \label{eq:multilevel_balance}
\end{equation}
where \(R_j^*\) is the bound on the initial distance at level \(j\),
computed from the preceding budgets \(K_1^*,\ldots,K_{j-1}^*\), with
\(R_1^*=R_1\).
The right-hand side is the final-level cost due to error that passes
through level \(j\). Increasing its budget reduces this contribution.
The later terms \(\Delta_j,\ldots,\Delta_{M-1}\) are added after
the factor \(B_j/K_j\), so they do not enter this derivative.

\section{Numerical Experiments}
\label{sec:experiments}

In this section, we test the two-level AGD allocation on regularized
logistic regression applied to HIGGS data \cite{whiteson2014}. The experiments
address the following questions.
\begin{enumerate}[leftmargin=2em]
  \item How well does a fixed budget based on the initial gradient
  bound perform, and why does it sometimes assign too many
  low-fidelity iterations?

  \item Can updating the cost comparison during iteration improve
  performance?
\end{enumerate}

\subsection{Logistic regression with data-fidelity levels}
\label{subsec:higgs_setup}

For each fidelity level \(m\in\{1,2\}\), define
\begin{equation}
  f^{(m)}(\vect{x})
  =
  \frac{1}{n_m}
  \sum_{i=1}^{n_m}
  \log\bigl(1+\exp(-b_i \vect{a}_i^\top \vect{x})\bigr)
  +
  \frac{\lambda}{2}\|\vect{x}\|^2,
  \qquad \vect{x}\in\mathbb R^d,
  \label{eq:experiment_logistic}
\end{equation}
where \(\vect{a}_i\in\mathbb R^d\) is a feature vector,
\(b_i\in\{-1,+1\}\) is its label, and \(\lambda=10^{-2}\) is the regularization parameter.
Labels originally coded as \(y_i\in\{0,1\}\) are recoded as
\(b_i=2y_i-1\). This gives the same loss as binary cross-entropy.

The high-fidelity objective uses a fixed stratified sample of
\(n_2=110{,}000\) HIGGS observations with \(d=28\) features. The
low-fidelity objective uses a fixed subset of \(n_1\) of these observations.
Both levels therefore share the same parameter space and regularization,
while their empirical losses and minimizers generally differ.

Full-gradient evaluation dominates the iteration cost and requires
\(O(n_m d)\) operations. Since \(d\) is fixed, we normalize costs by
the high-fidelity sample size. Define the sample ratio and the cost weights by
\begin{equation}
  r:=\frac{n_1}{n_2},
  \qquad c_1=r,
  \qquad c_2=1.
  \label{eq:experiment_cost_weights}
\end{equation}
The subsets remain fixed throughout each run, so both optimization
stages use deterministic full gradients.

We start from \(\vect{x}_0=0\) and test
\begin{equation}
  r\in\{0.05,0.15,0.25\},
  \qquad
  \varepsilon\in\{10^{-3},10^{-4},10^{-5}\}.
  \label{eq:experiment_parameters}
\end{equation}
The high-fidelity sample is selected using seed \(1\). Three stratified
subset sequences are generated using seeds \(11,12,13\). Stratification
preserves the class proportions up to integer rounding. For each seed,
the observations at \(r=0.05\) are contained in those at \(r=0.15\),
which are contained in those at \(r=0.25\).
For each feature, we subtract its mean in the high-fidelity sample and
divide by its sample standard deviation. The low-fidelity subsets use
these same standardized feature values.

This gives nine low-fidelity objectives and 27 combinations of subset
seed, sample ratio, and target accuracy. Varying \(r\) changes cost,
the distance between minimizers, and potentially the curvature constants. The
experiment therefore measures their combined effect.

\subsection{Model properties and applicability of the theory}
\label{subsec:higgs_model_properties}

Let \(\mat{X}_m\in\mathbb R^{n_m\times d}\) contain the feature vectors as
rows, and define
\[
  \sigma_i(\vect{x})
  =
  \frac{1}{1+\exp(b_i \vect{a}_i^\top \vect{x})}.
\]
The Hessian of \eqref{eq:experiment_logistic} is
\begin{equation}
  \nabla^2 f^{(m)}(\vect{x})
  =
  \frac{1}{n_m}
  \sum_{i=1}^{n_m}
  \sigma_i(\vect{x})\bigl(1-\sigma_i(\vect{x})\bigr)\vect{a}_i \vect{a}_i^\top
  +
  \lambda \mat{I}.
  \label{eq:experiment_hessian}
\end{equation}
Since \(0<\sigma_i(\vect{x})<1\), the scalar inequality
\(t(1-t)\leq1/4\) gives
\[
  0\leq\sigma_i(\vect{x})\bigl(1-\sigma_i(\vect{x})\bigr)\leq\frac14,
  \qquad \vect{a}_i \vect{a}_i^\top\succeq0.
\]
The matrix inequality follows from
\(\vect{v}^{\top}\vect{a}_i\vect{a}_i^{\top}\vect{v}
=(\vect{a}_i^{\top}\vect{v})^2\geq0\) for every vector \(\vect{v}\).
We therefore obtain the uniform bounds
\begin{equation}
  \lambda \mat{I}
  \preceq
  \nabla^2 f^{(m)}(\vect{x})
  \preceq
  \lambda \mat{I}+\frac{1}{4n_m}\mat{X}_m^\top \mat{X}_m
  \preceq
  L_m \mat{I},
  \label{eq:experiment_hessian_bounds}
\end{equation}
where we choose the valid constants
\begin{equation}
  \mu_m=\lambda,
  \qquad
  L_m
  =
  \lambda+
  \frac{\lambda_{\max}(\mat{X}_m^\top \mat{X}_m)}{4n_m}.
  \label{eq:experiment_constants}
\end{equation}
Thus, \(f^{(m)}\) is \(\mu_m\)-strongly convex and \(L_m\)-smooth even
when \(\mat{X}_m\) is rank deficient. These are bounds on the Hessian
eigenvalues, not claims that every Hessian attains them.
Quadratic regularization makes the objective tend to infinity as
\(\|\vect{x}\|\to\infty\), so each level has a unique minimizer.

The constants in \eqref{eq:experiment_constants} make
\Cref{ex:agd_rates} applicable, with \(A_m=2L_m\), \(p=2\),
\(B_m=\sqrt{4L_m/\mu_m}\), and \(q=1\). Taking
\(\Delta_1=\norm{\vect{x}^{*(1)}-\vect{x}^{*(2)}}\) in the theoretical model gives
\begin{equation}
 \norm{\vect{x}_{K_1}^{(1)}-\vect{x}^{*(2)}}
 \leq\frac{\sqrt{4L_1/\mu_1}}{K_1}\norm{\vect{x}_0-\vect{x}^{*(1)}}+\Delta_1,
 \qquad K_1\geq1.
 \label{eq:experiment_transfer}
\end{equation}
The allocation rules use these bounds \(L_m\) and \(\mu_m\). The distance
diagnostics use numerical approximations to the minimizers, as described below.

\subsection{Algorithm and theoretical allocation}
\label{subsec:experiment_algorithm}

At each level, AGD starts from
\(\vect{y}_0^{(m)}=\vect{x}_0^{(m)}\) and \(t_0=1\), and iterates
\begin{align}
  \vect{x}_{k+1}^{(m)}
  &=
  \vect{y}_k^{(m)}
  -
  \frac{1}{L_m}\nabla f^{(m)}(\vect{y}_k^{(m)}),\\
  t_{k+1}
  &=
  \frac{1+\sqrt{1+4t_k^2}}{2},\\
  \vect{y}_{k+1}^{(m)}
  &=
  \vect{x}_{k+1}^{(m)}
  +
  \frac{t_k-1}{t_{k+1}}
  \bigl(\vect{x}_{k+1}^{(m)}-\vect{x}_k^{(m)}\bigr).
\end{align}
The point \(\vect{y}_k^{(m)}\) extrapolates from the two most recent
iterates. This extrapolation is the momentum step. Resetting momentum
means setting \(\vect{y}_0^{(m)}=\vect{x}_0^{(m)}\) and \(t_0=1\)
again, so no extrapolation from the previous objective is carried over.
After \(K_1\) low-fidelity iterations, the high-fidelity stage starts
from \(\vect{x}_0^{(2)}=\vect{x}_{K_1}^{(1)}\), with its momentum reset.
The choice \(K_1=0\) denotes direct high-fidelity AGD from \(\vect{x}_0\).

For positive integer \(K_1\), the required high-fidelity budget from the convergence bound is
\begin{equation}
  \overline K_2(K_1)
  =
  \sqrt{\frac{2L_2}{\varepsilon}}
  \left(
    \frac{\sqrt{4L_1/\mu_1}}{K_1}
    \|\vect{x}_0-\vect{x}^{*(1)}\|
    +
    \Delta_1
  \right).
  \label{eq:experiment_terminal_bound}
\end{equation}
By Corollary~\ref{thm:optimal_allocation_no_shorthand}, the low-fidelity budget
that minimizes the cost model is
\begin{equation}
  K_1^*
  =
  \sqrt{
    \frac{1}{r}
    \sqrt{\frac{8L_1L_2}{\mu_1\varepsilon}}
    \|\vect{x}_0-\vect{x}^{*(1)}\|
  }.
  \label{eq:experiment_theoretical_budget}
\end{equation}

The theoretical budget involves an unknown distance to a minimizer.
The two allocation rules below replace it with a gradient bound.
The first fixes the budget at the initial point. The second updates the
decision after each low-fidelity iteration.

To study the distances in Figure~\ref{fig:higgs_sensitivity}, we also
compute numerical approximations \(\vect{x}_{\mathrm{ref}}^{(m)}\) to
the minimizers, using damped Newton iterations until
\[
 \|\nabla f^{(m)}(\vect{x}_{\mathrm{ref}}^{(m)})\|\leq10^{-8}.
\]
By Proposition~\ref{prop:certificates} and \(\mu_m=10^{-2}\), these approximations
satisfy
\[
 \|\vect{x}_{\mathrm{ref}}^{(m)}-\vect{x}^{*(m)}\|\leq10^{-6},
 \qquad
 f^{(m)}(\vect{x}_{\mathrm{ref}}^{(m)})-f^{(m)}(\vect{x}^{*(m)})
 \leq5\times10^{-15}.
\]
The triangle inequality therefore bounds the error in a measured
distance to one minimizer by \(10^{-6}\), and the error in the measured
distance between two minimizers by \(2\times10^{-6}\). We use the
measured distances directly in the diagnostic plots; adding these
tolerances would give rigorous upper bounds. These numerical
approximations are used only for the plots and checks of objective
gaps. Their computation is excluded from the reported cost, and they
are not used to choose budgets or to stop the algorithms.

\subsection{Evaluation protocol}
\label{subsec:experiment_evaluation}

Every budget strategy in \Cref{subsec:budget_strategies} is evaluated
with the same high-fidelity stopping rule and the same cost measure,
defined here.

\subsubsection{Stopping criterion and observed cost}
We use the same stopping rule for every high-fidelity stage, with the
iteration count defined by
\begin{equation}
 K_{2,\mathrm{cert}}(K_1;\varepsilon)
 :=\min\left\{k\in\mathbb Z_{\geq0}:
 \frac{\|\nabla f^{(2)}(\vect{x}_k^{(2)})\|^2}{2\mu_2}
 \leq\varepsilon\right\}.
 \label{eq:experiment_first_hit}
\end{equation}
The test is checked from \(k=0\) and certifies the target accuracy
without a reference optimum.

We define the \emph{observed cost} by
\begin{equation}
 \mathrm{Cost}_{\mathrm{obs}}(K_1;\varepsilon)
 :=c_1K_1+c_2K_2,
 \qquad K_2=K_{2,\mathrm{cert}}(K_1;\varepsilon).
 \label{eq:experiment_observed_cost}
\end{equation}
Here \(K_1\) and \(K_2\) count completed iterations, and
\(c_1=r\), \(c_2=1\) follow \eqref{eq:experiment_cost_weights}.
The observed cost does not separately charge the gradients evaluated
for the switching and stopping tests. If each completed iteration is
also charged one gradient evaluation for these tests, the costs of online switching and direct AGD both double, so
their ratios are unchanged. The initial fixed budget and the best
tested budget, which need no switching test, then have median cost
ratios of \(1.504\) and \(0.389\).

\subsubsection{Performance measure}
For a budget \(K_1\), we report its observed cost relative to the direct
AGD cost
\(\mathrm{Cost}_{\mathrm{dir,obs}}(\varepsilon)
:=\mathrm{Cost}_{\mathrm{obs}}(0;\varepsilon)\),
\begin{equation}
 \mathrm{Cost\ ratio}
 =\frac{\mathrm{Cost}_{\mathrm{obs}}(K_1;\varepsilon)}
 {\mathrm{Cost}_{\mathrm{dir,obs}}(\varepsilon)}.
 \label{eq:experiment_selected_speedup}
\end{equation}
A ratio below one means that the budget costs less than direct AGD.
For example, a ratio of \(0.5\) means half the cost. We compute this
ratio for each configuration before taking the median across all 27
configurations. We also count how many configurations have a ratio
strictly below one.

\subsection{Budget strategies}
\label{subsec:budget_strategies}

We compare direct AGD with two choices of the low-fidelity budget.
The \emph{initial fixed budget} is computed once from the gradient at
\(\vect{x}_0\), as defined in \Cref{subsec:initial_certificate_allocation}.
\emph{Online switching} updates the decision during the low-fidelity
iterations, as defined in \Cref{subsec:online_cost_comparison}.
For the purpose of benchmarking, we also report the \emph{best tested
budget}, defined in \Cref{subsec:best_tested_budget} and selected after
all tested budgets have been evaluated.

\subsubsection{Best tested budget}
\label{subsec:best_tested_budget}
For each subset seed and sample ratio, let \(\mathcal K\) be the union
of tested budgets over the three targets. It includes zero, powers of
two, the fixed and online allocations, and budgets from exploratory
runs and local refinement. The complete set of tested budgets is
included in the reproduction archive (see Data and code availability).
The best tested budget is
\begin{equation}
 \widehat K_1^*
 \in\operatorname*{arg\,min}_{K_1\in\mathcal K}
 \mathrm{Cost}_{\mathrm{obs}}(K_1;\varepsilon).
 \label{eq:experiment_empirical_budget}
\end{equation}
This benchmark uses information from all tested budgets, excludes
the cost of searching them, and need not be optimal over all integer budgets.

\subsubsection{Initial fixed budget from the initial gradient}
\label{subsec:initial_certificate_allocation}

Substituting \(\|\nabla f^{(1)}(\vect{x}_0)\|/\mu_1\) for
\(\|\vect{x}_0-\vect{x}^{*(1)}\|\) in
\eqref{eq:experiment_theoretical_budget} gives the initial fixed budget
\begin{equation}
 K_{1,\mathrm{init}}
 =\left\lceil
 \sqrt{\frac1r\sqrt{\frac{8L_1L_2}{\mu_1\varepsilon}}
 \frac{\|\nabla f^{(1)}(\vect{x}_0)\|}{\mu_1}}
 \right\rceil.
 \label{eq:experiment_initial_certificate_budget}
\end{equation}
The initial fixed-budget algorithm is as follows.
\begin{algorithm}[Initial fixed budget]
\label{alg:initial_fixed_budget}
Initialize low-fidelity AGD at \(\vect{x}_0\).
\begin{enumerate}[leftmargin=2em]
 \item Compute \(K_{1,\mathrm{init}}\) from
 \eqref{eq:experiment_initial_certificate_budget}.
 \item Perform exactly \(K_{1,\mathrm{init}}\) low-fidelity AGD
 iterations. If the initial gradient is zero, this budget is zero
 and the low-fidelity stage is skipped.
 \item Use \(\vect{x}_{K_{1,\mathrm{init}}}^{(1)}\) as the high-fidelity
 initial point, reset high-fidelity momentum, and run until
 \eqref{eq:experiment_first_hit} is satisfied.
\end{enumerate}
\end{algorithm}

Table~\ref{tab:higgs_initial_fixed} reports the initial fixed budgets
and their observed cost ratios. The budgets range from 217 to 1714
iterations, and every cost ratio exceeds one.

The dependence on target accuracy follows the model prediction in
\eqref{eq:agd_K1_expanded}. With the sample ratio and subset seed fixed,
the initial budget before rounding is proportional to
\(\varepsilon^{-1/4}\). Reducing \(\varepsilon\) by a factor of ten
therefore increases this budget by \(10^{1/4}\approx1.78\), as seen in
Table~\ref{tab:higgs_initial_fixed}. We next examine why these budgets
assign so much low-fidelity work.

\begin{table}[!htbp]
 \centering
 \caption{Initial fixed budgets and observed costs relative to direct
 AGD. Each entry gives the minimum--maximum range over subset seeds
 11, 12, and 13. Cost ratios follow
 \eqref{eq:experiment_selected_speedup}.}
 \label{tab:higgs_initial_fixed}
 \small\renewcommand{\arraystretch}{1.2}
 \begin{tabular*}{\linewidth}{@{\extracolsep{\fill}}ccrr@{}}
 \toprule
 \(r\) & \(\varepsilon\) & \(K_{1,\mathrm{init}}\) & Cost ratio \\
 \midrule
 0.05 & $10^{-3}$ & 470--542 & 1.703--1.894 \\
 0.05 & $10^{-4}$ & 836--964 & 1.522--1.728 \\
 0.05 & $10^{-5}$ & 1486--1714 & 2.238--2.540 \\
 \addlinespace
 0.15 & $10^{-3}$ & 277--292 & 2.642--2.711 \\
 0.15 & $10^{-4}$ & 492--519 & 2.356--2.424 \\
 0.15 & $10^{-5}$ & 875--923 & 3.577--3.696 \\
 \addlinespace
 0.25 & $10^{-3}$ & 217--223 & 3.236--3.375 \\
 0.25 & $10^{-4}$ & 385--397 & 2.924--3.035 \\
 0.25 & $10^{-5}$ & 684--705 & 4.381--4.625 \\
 \bottomrule
 \end{tabular*}
\end{table}

\begingroup
\setlength{\intextsep}{8pt plus 2pt minus 2pt}
\begin{figure}[p]
 \centering
 \includegraphics[width=119mm]{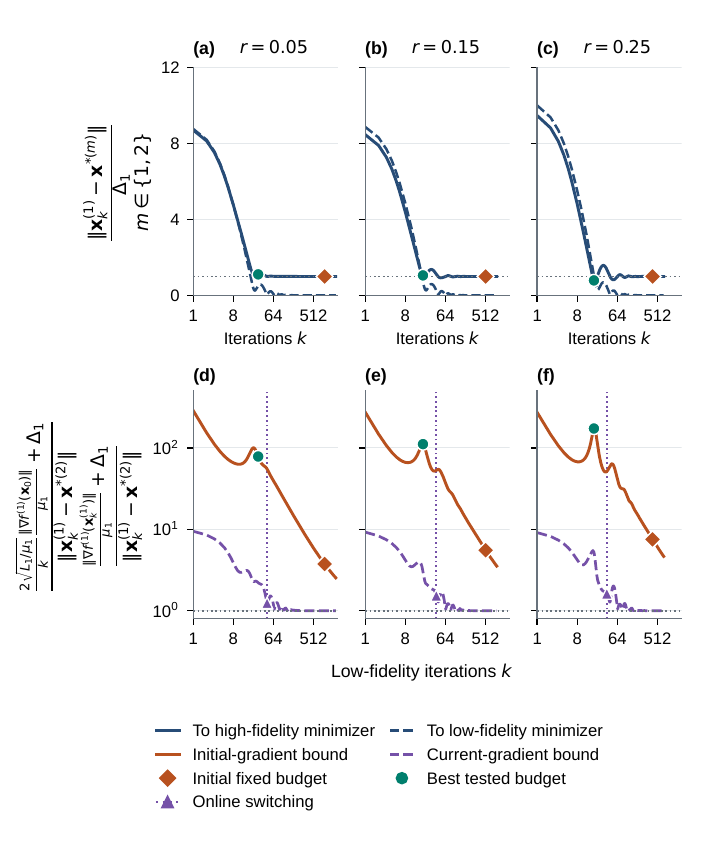}
 \caption{Distances and their bounds during low-fidelity AGD, for seed 11.
 Each horizontal axis counts low-fidelity iterations \(k\), and each
 column uses one sample ratio \(r\). The upper panels show distances
 to the two minimizers, divided by their distance \(\Delta_1\).
 The lower panels divide the initial-gradient bound
 \eqref{eq:experiment_initial_transfer_certificate} and the
 current-gradient bound
 \eqref{eq:experiment_current_transfer_certificate} by the measured
 distance to the high-fidelity minimizer. A lower-panel value of one
 means the bound equals that distance. Diamonds, triangles, and circles
 mark the initial fixed budget, the online switching position of
 \Cref{subsec:online_cost_comparison}, and the best tested budget at
 \(\varepsilon=10^{-4}\). Dotted vertical lines repeat the online
 switching positions.
 Horizontal axes and lower vertical axes are logarithmic. Distances
 are measured using the numerical minimizer approximations.}
 \label{fig:higgs_sensitivity}
\end{figure}
\endgroup

\subsubsection{Why the initial fixed budget is too large}
\label{subsec:distance_diagnostics}
Figure~\ref{fig:higgs_sensitivity} follows the low-fidelity AGD iterates
\(\vect{x}_k^{(1)}\) for subset seed 11, with one column for each sample
ratio \(r\). The horizontal axis counts completed low-fidelity
iterations \(k\). The upper panels plot the distances
\(\|\vect{x}_k^{(1)}-\vect{x}^{*(2)}\|\) and
\(\|\vect{x}_k^{(1)}-\vect{x}^{*(1)}\|\) to the high- and
low-fidelity minimizers, both divided by
\(\Delta_1=\|\vect{x}^{*(1)}-\vect{x}^{*(2)}\|\). The lower panels
plot two upper bounds on \(\|\vect{x}_k^{(1)}-\vect{x}^{*(2)}\|\),
defined in \eqref{eq:experiment_initial_transfer_certificate} and
\eqref{eq:experiment_current_transfer_certificate} below, each divided
by the measured distance, so a value of one means that the bound equals
the distance.

The upper panels show why further low-fidelity iterations can bring
little improvement to the high-fidelity starting point. The distance to
the high-fidelity minimizer has nearly stabilized while the iterates
continue to approach the low-fidelity minimizer. These distances
approach \(\Delta_1\) and zero, respectively, so the normalized
limiting high-fidelity distance equals one for all three sample ratios.

The lower panels compare the two upper bounds
\eqref{eq:experiment_initial_transfer_certificate} and
\eqref{eq:experiment_current_transfer_certificate} on the distance to
the high-fidelity minimizer. The bound used by the initial fixed budget is
\begin{equation}
 \|\vect{x}_k^{(1)}-\vect{x}^{*(2)}\|
 \leq
 \frac{2\sqrt{L_1/\mu_1}}{k}
 \frac{\|\nabla f^{(1)}(\vect{x}_0)\|}{\mu_1}
 +\Delta_1,\qquad k\geq1.
 \label{eq:experiment_initial_transfer_certificate}
\end{equation}
It combines the gradient bound at the initial point with the AGD
convergence bound; we call its right-hand side the
\emph{initial-gradient bound}. Applying the gradient bound at the
current point instead gives
\begin{equation}
 \|\vect{x}_k^{(1)}-\vect{x}^{*(2)}\|
 \leq
 \frac{\|\nabla f^{(1)}(\vect{x}_k^{(1)})\|}{\mu_1}
 +\Delta_1,
 \label{eq:experiment_current_transfer_certificate}
\end{equation}
whose right-hand side we call the \emph{current-gradient bound}.
Both bounds use the same \(\Delta_1\) and triangle inequality. Only the
estimate of the distance to the low-fidelity minimizer changes.

The initial-gradient bound remains large after the measured
distance has stabilized, and its \(1/k\) term keeps decreasing. The cost
model counts this decrease as a saving in high-fidelity work and
therefore assigns too large a budget.

The current-gradient bound approaches the measured distance much
earlier. The current gradient thus gives a better indication of when
further low-fidelity iterations offer little improvement to the
high-fidelity initial point. The online switching rule introduced in
\Cref{subsec:online_cost_comparison} uses the current gradient to
switch near this region. The vertical lines and
triangles mark where it switches at \(\varepsilon=10^{-4}\), at
\(k=46\), \(40\), and \(37\) for the three sample ratios. These
positions are much closer to the best tested budgets, shown by circles,
than the initial fixed budgets are.

\subsubsection{Online switching by comparing remaining costs}
\label{subsec:online_cost_comparison}
\begingroup
\setlength{\abovedisplayskip}{9pt plus 2pt minus 2pt}
\setlength{\belowdisplayskip}{9pt plus 2pt minus 2pt}
\setlength{\abovedisplayshortskip}{3pt plus 2pt}
\setlength{\belowdisplayshortskip}{6pt plus 2pt minus 2pt}
At each low-fidelity iterate, the method can switch to the
high-fidelity objective now or perform additional low-fidelity
iterations and switch later. We compare the model cost of the
remaining work under these two options and switch as soon as switching now
is no more expensive than the best continuation. Work already done is
common to both options and does not enter the comparison. The current gradient provides an
estimate of the remaining distance to the low-fidelity minimizer,
from which we construct both cost models.

At the current low-fidelity iterate \(\vect{x}_k^{(1)}\), define the
gradient bound on its distance to the low-fidelity minimizer and a
real-valued budget for additional low-fidelity iterations by
\begin{align}
 d_k&:=\frac{\|\nabla f^{(1)}(\vect{x}_k^{(1)})\|}{\mu_1},
 \label{eq:experiment_online_distance}\\
 K_1^*(\vect{x}_k^{(1)})
 &:=\sqrt{\frac{c_2}{c_1}\sqrt{\frac{8L_1L_2}{\mu_1\varepsilon}}d_k}.
 \label{eq:experiment_online_budget}
\end{align}

The cost models for switching immediately and for restarting
low-fidelity AGD for an additional real budget \(s>0\) are
\begin{equation}
 \mathrm{Cost}_{\mathrm{switch}}(\vect{x}_k^{(1)})
 =c_2\sqrt{\frac{2L_2}{\varepsilon}}(d_k+\Delta_1),
 \label{eq:experiment_switch_cost}
\end{equation}
\begin{equation}
 \mathrm{Cost}_{\mathrm{continue}}(s;\vect{x}_k^{(1)})
 =c_1s+c_2\sqrt{\frac{2L_2}{\varepsilon}}
 \left(\frac{2\sqrt{L_1/\mu_1}}{s}d_k+\Delta_1\right).
 \label{eq:experiment_continue_cost}
\end{equation}
The continuation model \eqref{eq:experiment_continue_cost} is the
two-level cost model \eqref{eqn:warmstartcost} of
Corollary~\ref{thm:optimal_allocation_no_shorthand} with the initial point
\(\vect{x}_0\) replaced by \(\vect{x}_k^{(1)}\) and the distance
\(\norm{\vect{x}_0-\vect{x}^{*(1)}}\) replaced by its bound \(d_k\).
For \(d_k>0\), Corollary~\ref{thm:optimal_allocation_no_shorthand} therefore
gives its unique minimizer \(s=K_1^*(\vect{x}_k^{(1)})\), which is
\eqref{eq:experiment_online_budget}, and by
\eqref{eq:Cost_star_no_shorthand} the minimum value
\begin{equation}
 \min_{s>0}\mathrm{Cost}_{\mathrm{continue}}(s;\vect{x}_k^{(1)})
 =2c_1K_1^*(\vect{x}_k^{(1)})
 +c_2\sqrt{\frac{2L_2}{\varepsilon}}\Delta_1.
 \label{eq:experiment_min_continue_cost}
\end{equation}

The model favors immediate switching exactly when
\begin{equation}
 \begin{aligned}
 &\mathrm{Cost}_{\mathrm{switch}}(\vect{x}_k^{(1)})
 \leq\min_{s>0}\mathrm{Cost}_{\mathrm{continue}}(s;\vect{x}_k^{(1)})\\
 &\quad\Longleftrightarrow\quad
 c_2\sqrt{\frac{2L_2}{\varepsilon}}d_k
 \leq2c_1K_1^*(\vect{x}_k^{(1)}).
 \end{aligned}
 \label{eq:experiment_online_switch}
\end{equation}
The term \(c_2\sqrt{2L_2/\varepsilon}\,\Delta_1\) appears in both costs
and cancels, so the comparison needs neither \(\Delta_1\) nor either
minimizer. For \(d_k>0\), substituting \eqref{eq:experiment_online_budget} gives
\begin{equation}
 \begin{aligned}
 &\frac{c_1[K_1^*(\vect{x}_k^{(1)})]^2}{2\sqrt{L_1/\mu_1}}
 \leq2c_1K_1^*(\vect{x}_k^{(1)})\\
 &\quad\Longleftrightarrow\quad
 K_1^*(\vect{x}_k^{(1)})\leq4\sqrt{L_1/\mu_1}.
 \end{aligned}
 \label{eq:experiment_online_threshold}
\end{equation}
If \(d_k=0\), the low-fidelity minimizer has already been reached and
the rule switches immediately.
At equality in \eqref{eq:experiment_online_threshold}, the
cost-minimizing real budget for a restarted continuation is
\(s=4\sqrt{L_1/\mu_1}\), and its distance-bound term in
\eqref{eq:experiment_continue_cost} equals \(d_k/2\). For \(d_k>0\),
the rule therefore switches exactly when the cost-minimizing restarted
continuation would leave a distance-bound term of at least \(d_k/2\),
that is, when it would remove at most half of the current bound. By
\eqref{eq:experiment_online_distance} and
\eqref{eq:experiment_online_budget}, the condition
\eqref{eq:experiment_online_threshold} is also equivalent to
\(\|\nabla f^{(1)}(\vect{x}_k^{(1)})\|^2/(2\mu_1)
\leq16(c_1/c_2)^2(L_1/L_2)\,\varepsilon\), which covers \(d_k=0\) as
well.

The online switching algorithm is as follows.
\begin{algorithm}[Online switching]
\label{alg:online_cost_comparison}
Initialize low-fidelity AGD at \(\vect{x}_0\).
\begin{enumerate}[leftmargin=2em]
 \item At each low-fidelity iterate
 \(\vect{x}_k^{(1)}\), compute \(d_k\) and
 \(K_1^*(\vect{x}_k^{(1)})\). Stop the low-fidelity stage at the first
 iterate at which \(K_1^*(\vect{x}_k^{(1)})\leq4\sqrt{L_1/\mu_1}\), the
 switching condition in \eqref{eq:experiment_online_threshold}, or at
 which \(d_k=0\). Otherwise,
 perform one more low-fidelity AGD iteration, retaining its momentum,
 and repeat the comparison.
 \item Use the final low-fidelity iterate as the high-fidelity initial
 point, reset high-fidelity momentum, and run until
 \eqref{eq:experiment_first_hit} is satisfied.
\end{enumerate}
\end{algorithm}

The comparison uses the current gradient and known curvature
constants. It judges low-fidelity accuracy through the cost model and
does not directly measure changes in the distance to the high-fidelity
minimizer. Its derivation assumes that low-fidelity AGD would restart
at the current point, whereas the implementation retains the existing
momentum when it continues. The rule is therefore a heuristic based
on the cost model. It provides an estimate of the switching position,
and we assess its performance using observed costs.
\endgroup

\subsection{Allocation performance}
\label{subsec:higgs_results}

Table~\ref{tab:higgs_initial_online} compares the initial fixed budget
and online switching using the same groups as
Table~\ref{tab:higgs_initial_fixed}. Online switching uses fewer
low-fidelity iterations and has a lower observed cost in every
configuration.

\begin{table}[!htbp]
 \centering
 \caption{Initial fixed budgets and online switching. Each entry gives
 the minimum--maximum range over subset seeds 11, 12, and 13.
 \(K_1\) counts completed low-fidelity iterations. Cost ratios are
 relative to direct AGD, as defined in
 \eqref{eq:experiment_selected_speedup}.}
 \label{tab:higgs_initial_online}
 \small\setlength{\tabcolsep}{4pt}\renewcommand{\arraystretch}{1.2}
 \begin{tabular*}{\linewidth}{@{\extracolsep{\fill}}ccrrrr@{}}
 \toprule
 & & \multicolumn{2}{c}{Initial fixed budget}
 & \multicolumn{2}{c}{Online switching}\\
 \cmidrule(lr){3-4}\cmidrule(l){5-6}
 \(r\) & \(\varepsilon\) & \(K_1\) & Cost ratio
 & \(K_1\) & Cost ratio\\
 \midrule
 0.05 & $10^{-3}$ & 470--542 & 1.703--1.894 & 39--41 & 0.447--0.608 \\
 0.05 & $10^{-4}$ & 836--964 & 1.522--1.728 & 44--46 & 0.342--0.450 \\
 0.05 & $10^{-5}$ & 1486--1714 & 2.238--2.540 & 63--66 & 0.388--0.694 \\
 \addlinespace
 0.15 & $10^{-3}$ & 277--292 & 2.642--2.711 & 20--20 & 0.500--0.500 \\
 0.15 & $10^{-4}$ & 492--519 & 2.356--2.424 & 39--40 & 0.385--0.500 \\
 0.15 & $10^{-5}$ & 875--923 & 3.577--3.696 & 61--62 & 0.575--0.674 \\
 \addlinespace
 0.25 & $10^{-3}$ & 217--223 & 3.236--3.375 & 18--18 & 0.528--0.583 \\
 0.25 & $10^{-4}$ & 385--397 & 2.924--3.035 & 34--37 & 0.431--0.646 \\
 0.25 & $10^{-5}$ & 684--705 & 4.381--4.625 & 42--43 & 0.589--0.685 \\
 \bottomrule
 \end{tabular*}
\end{table}

Table~\ref{tab:higgs_allocation_summary} compares the two allocation
rules with the best tested budget. The initial fixed budget costs more
than direct AGD in every configuration. Online switching costs less
in all 27 configurations, with a median cost ratio of \(0.528\).
The best tested budget gives a median ratio of \(0.463\), using
information from the complete tested set.

\begin{table}[!htbp]
 \centering
 \caption{Observed cost across the 27 configurations. The cost ratio
 is method cost divided by direct AGD cost. A ratio below one means
 less cost than direct AGD. The best tested budget is a retrospective
 benchmark.}
 \label{tab:higgs_allocation_summary}
 \small\setlength{\tabcolsep}{5pt}\renewcommand{\arraystretch}{1.2}
 \begin{tabular*}{\linewidth}{@{\extracolsep{\fill}}lrr@{}}
 \toprule
 Method or benchmark & \shortstack{Median cost\\ratio}
 & \shortstack{Configurations\\with less cost}\\
 \midrule
 Initial fixed budget & 2.694 & 0/27 \\
 Online switching & 0.528 & 27/27 \\
 Best tested budget & 0.463 & 27/27 \\
 \bottomrule
 \end{tabular*}
\end{table}

\newpage
Figure~\ref{fig:higgs_allocation_cost} compares online switching with
the best tested budget at \(\varepsilon=10^{-4}\).

In all three panels, online switching stops in the same early-iteration
region as the best tested budget. The initial fixed budgets lie much
farther to the right, where additional low-fidelity work increases
total cost. The online and best-tested costs are particularly close
at \(r=0.05\); the gap is larger at \(r=0.15\) and \(r=0.25\).
Thus, online switching finds a switching region close to the best
tested budget without evaluating future budgets.

\begin{figure}[!htbp]
 \centering
 \includegraphics[width=119mm]{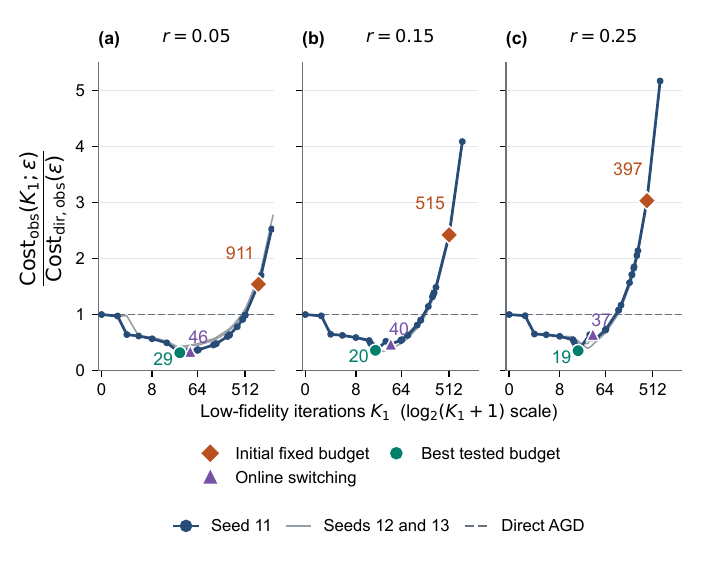}
 \caption{Observed cost relative to direct AGD at \(\varepsilon=10^{-4}\).
 The curve with dots is seed 11, and the thin curves are seeds 12 and
 13. Diamonds, triangles, and
 large circles mark the initial fixed budget, online switching, and
 best tested budget, labeled by their low-fidelity iteration counts.
 All costs follow \eqref{eq:experiment_observed_cost}. Values below
 one improve on direct AGD. The horizontal coordinate is
 \(\log_2(K_1+1)\), with ticks showing \(K_1\).}
 \label{fig:higgs_allocation_cost}
\end{figure}

\FloatBarrier
\section{Conclusion}
\label{sec:conclusion}

Low-fidelity iterations are worthwhile when they reduce the total cost
of reaching the required high-fidelity accuracy. Additional iterations
at one level incur a cost but can improve the starting point for the
next level. We studied how to
choose the number of iterations at each level
before switching to a more accurate objective. For smooth, strongly
convex objectives, convergence bounds give a model of this trade-off.
The model has a unique minimizing budget for two levels when the
initial low-fidelity error is nonzero. For AGD, the two-level budgets
have closed forms, and the multi-level problem can be written as a
strictly convex optimization problem.
The two terms in the bound \eqref{eq:warmstart_bound_poly} play roles
similar to the algebraic and discretization errors in adaptive finite
element methods, where the iterative solver on each mesh can be
stopped once a norm of the algebraic residual is bounded by a multiple
of an a posteriori estimate of the discretization
error~\cite{arioli2013}. Unlike that stopping criterion, which compares
only the two errors, our allocation also accounts for the iteration
costs of the two levels.

The theoretical budgets depend on unknown distances to minimizers,
which we estimate using gradient bounds. Conservative initial
estimates motivate updating the switching decision using the current
gradient. In the HIGGS experiment, budgets computed from the initial
gradient bound spend too
many iterations on the low-fidelity objective and never improve on
direct AGD in the tested configurations. Online switching uses the
current gradient to update the cost comparison and reduces cost in all
27 configurations. Its median cost is 52.8\% of direct AGD cost,
compared with 46.3\% for the best tested budget. This supports the use
of updated information, while also showing that minimizing a
conservative cost model need not minimize observed work.
In Figure~\ref{fig:higgs_sensitivity}, the initial-gradient bound
keeps decreasing after further low-fidelity iterations have nearly
stopped improving the high-fidelity starting point, so a bound that guarantees
accuracy need not predict the value of further computation.

Further experiments are needed to assess how well the allocation
strategy generalizes across optimization problems. Future work should
also consider how to combine it with different optimization methods.
For AGD, this includes using momentum information in the switching
decision and accounting for the stopping rule when estimating the
remaining cost. More generally, future work should identify computable
predictors of how further low-fidelity iterations change the
high-fidelity stopping time in \eqref{eq:experiment_first_hit}.
Another direction is derivative-free optimization,
where the allocation would need a cost model based on function
evaluations and a computable measure of progress that does not require
gradients.

\backmatter
\section*{Statements and Declarations}

\subsection*{Funding}
This research was supported by the Natural Sciences and Engineering
Research Council of Canada (NSERC), Discovery Grant \#2023-03555,
and the University of British Columbia.

\subsection*{Competing interests}
The authors have no relevant financial or non-financial interests to disclose.

\subsection*{Author contributions}
All authors contributed to the study conception and design.
Material preparation, data collection, and analysis were performed by
Zhongda Huang. The first draft of the manuscript was written by
Zhongda Huang, and all authors commented on previous versions of the
manuscript. All authors read and approved the final manuscript.

\subsection*{Data and code availability}
The source HIGGS data are publicly available from the UCI Machine
Learning Repository~\cite{whiteson2014}. The processed sample, experiment code, and saved numerical results
used in this study are available in the
\href{https://github.com/ViggleH/multifidelity-warmstart/releases/tag/manuscript-2026-09-11}{versioned GitHub reproduction archive}
(\texttt{manuscript-2026-09-11}).

\subsection*{Use of generative AI tools}
During the preparation of this work, the authors used generative AI
tools from OpenAI (ChatGPT, Codex) and Anthropic (Claude) to assist with
coding, data analysis, checking derivations and citations, drafting and
editing text, and manuscript consistency checks. AI tools were not used
to generate data or figures. The authors reviewed and verified all
content and take full responsibility for the final publication.

\begingroup
\interlinepenalty=10000
\bibliography{references}
\endgroup
\end{document}